\documentclass[9 pt,reqno]{amsart}

\usepackage{amsmath, amsfonts, amssymb, amscd, enumerate, color}
\theoremstyle{plain}
\newtheorem{theo}{Theorem}[section]
\newtheorem{lem}[theo]{Lemma}

\theoremstyle{definition}

\theoremstyle{plain}
\newtheorem{lemma}[theo]{Lemma}
\newtheorem{theorem}[theo]{Theorem}

\theoremstyle{definition}

\newtheorem{remark}[theo]{Remark}

\newcommand{\beq}{\begin{equation}}
\newcommand{\eeq}{\end{equation}}

\newcommand{\bC}{\mathbb{C}}

\newcommand{\gc}{\mathfrak{c}}

\newcommand{\gf}{\mathfrak{f}}
\renewcommand{\gg}{\mathfrak{g}}
\newcommand{\gh}{\mathfrak{h}}
\newcommand{\gk}{\mathfrak{k}}

\newcommand{\gm}{\mathfrak{m}}

\newcommand{\gp}{\mathfrak{p}}
\newcommand{\gq}{\mathfrak{q}}

\newcommand{\gs}{\mathfrak{s}}
\newcommand{\gt}{\mathfrak{t}}
\newcommand{\gu}{\mathfrak{u}}

\newcommand{\gz}{\mathfrak{z}}

\newcommand{\so}{\mathfrak{so}}
\newcommand{\su}{\mathfrak{su}}

\newcommand\SO{\mathrm{SO}}
\newcommand\SU{\mathrm{SU}}
\newcommand\U{\mathrm{U}}

\renewcommand\sp{\mathfrak{sp}}

\renewcommand{\square}{\kern1pt\vbox
{\hrule height 0.6pt\hbox{\vrule width 0.6pt\hskip 3pt
\vbox{\vskip 6pt}\hskip 3pt\vrule width 0.6pt}\hrule height0.6pt}\kern1pt}

\DeclareMathOperator\End{End\;}

\DeclareMathOperator\Ad{Ad}
\DeclareMathOperator\ad{ad}

\newcommand{\be}{\begin{equation}}
\newcommand{\ee}{\end{equation}}

\def\<#1,#2>{\langle\,#1,\,#2\,\rangle}
\newcommand{\arr}{\begin{array}{rlll}}
\newcommand{\ea}{\end{array}}
\newcommand{\bea}{\begin{eqnarray}}
\newcommand{\eea}{\end{eqnarray}}
\newcommand{\bean}{\begin{eqnarray*}}
\newcommand{\eean}{\end{eqnarray*}}

\def\sideremark#1{\ifvmode\leavevmode\fi\vadjust{%            The remark
\vbox to0pt{\hbox to 0pt{\hskip\hsize\hskip1em%               will appear only
\vbox{\hsize3cm\tiny\raggedright\pretolerance10000%          on the side
\noindent #1\hfill}\hss}\vbox to8pt{\vfil}\vss}}}%           in 3cm
\newcounter{ssig}
\newcounter{ttig}
\newtheorem*{notn*}{Notation}

\title[The index of symmetry of a flag manifold] {The index of symmetry of a flag manifold}

\author{Fabio Podest\`a}

\subjclass[2010]{53C30, 53C35}
\keywords{Homogeneous manifolds, flag manifolds, symmetric spaces.}

\begin{document}
\begin{abstract} {We study the  index of symmetry of a compact generalized flag manifold $M=G/H$ endowed with an invariant K\"ahler structure. When the group $G$ is simple we  show that the leaves of  symmetry  are irreducible Hermitian symmetric spaces and we estimate their dimension.}\end{abstract}

\maketitle
\section{Introduction}
Recently Olmos, Reggiani and Tamaru (\cite{ORT}) introduced the notion of index of symmetry for a Riemannian manifold. In particular given a Riemannian manifold $(M,g)$ on which the full group of isometries $G$ acts transitively, one can define the {\it distribution of symmetry} $\mathcal D$ which at each point $x\in M$ is given by the values at $x$ of all Killing vector fields $X$ with $\nabla X|_x=0$, where $\nabla$ denotes the Levi Civita connection of $g$. The main motivation for introducing $\mathcal D$ is the fact that the {\it index of symmetry} $\imath_s(M):= \dim \mathcal D$ coincides with $\dim M$ precisely when $(M,g)$ is a symmetric space. Therefore the coindex $\dim M - \imath_s(M)$ can be viewed as a sort of "distance" of the manifold from being a Riemannian symmetric space. It is proved that the distribution $\mathcal D$ is actually integrable and that the maximal integral submanifolds, which we will call the {\it leaves of symmetry}, are Riemannian symmetric spaces which are embedded into $M$ as totally geodesic submanifolds.\par
In this work we will focus on compact generalized flag manifolds, namely compact homogeneous spaces $G/H$ where $G$ is a connected compact semisimple Lie group and $H$ is the centralizer in $G$ of a torus. It is well known that these manifolds exhaust all compact K\"ahler manifolds which admit a transitive semisimple compact Lie group of biholomorphic isometries. If $M=G/H$ is a such a manifold endowed with an invariant K\"ahler structure given by an invariant complex structure $J$ and a K\"ahler metric g compatible with $J$, then $(M,g)$ is de Rham irreducible if and only if $G$ is simple and in this case $G$ coincides (up to covering) with the full group of isometries of $(M,g)$ up to a few exceptions which can be listed (see \S 2). Our main result can be stated as follows.

\begin{theorem}\label{main} Let $M=G/H$ be a compact generalized flag manifold endowed with a non-symmetric $G$-invariant K\"ahler structure $(g,J)$. Suppose that $G$ is simple and that it coincides with the full group of isometries of $g$. Then
\begin{itemize}
\item[i)] the leaves of symmetry are complex , totally geodesic submanifolds which  are irreducible Hermitian symmetric spaces.
\item[ii)] There exists a compact Lie subgroup $H'\supset H$ such that the fibers of the fibration $G/H\to G/H'$ coincide with the leaves of symmetry. The subgroup $H'$ depends only on the complex structure $J$ and not on the metric $g$.
\item[iii)] If $k$ denotes the co-index of symmetry (i.e. $k=\dim G/H'$), then
\begin{equation}\label{est}\dim G \leq \frac 12\ k(k-1) \quad \mathrm{and}\quad k\geq 6\end{equation}
with $k=6$ precisely when $\gg=\su(4)$ and $\gh = 2\mathbb R\oplus \su(2)$.
\end{itemize}
\end{theorem}
We first remark that when $G$ is not simple and splits (locally) as a product $G = G_1\times\ldots\times G_r$ of simple factors $G_1,\ldots,G_r$, then there are
subgroups $H_i\subset G_i$ such that $M = G_1/H_1\times \ldots G_r/H_r$ biholomorphically and isometrically , so that the above theorem can be applied on each factor leading to the general description of the leaves of symmetry for any generalized flag manifold. The hypothesis that $G$ coincides (up to covering) with the full group of isometries is also not too restrictive, as the result due to Onishchick shows (see Theorem \ref{oni}).\par
As stated in our main theorem, it is remarkable that the leaves of symmetry are {\it irreducible} symmetric spaces and that the subgroup $H'$ depends exclusively on the invariant complex structure and not on the K\"ahler metric. In the Remark \ref{dynkin} we give a simple and constructive way to identify the subgroup $H'$ in terms of the highest root of the root system of the Lie algebra $\gg$ of $G$, equipped with the ordering corresponding to the invariant complex structure. Note also that in general the invariant complex structure on $M$ does {\it not} descend to an invariant (almost)-complex structure on $G/H'$. \par
Note also that in \cite{BOR} the authors prove an estimate for the co-index of symmetry from which we get $\dim G\leq \frac 12 k(k+1)$, while the estimate
\eqref{est} is finer in our setting. Also the fact that $k\geq 6$ improves the inequality $k\geq 2$ which is proved in a general setting in \cite{BOR}.\par

The paper is structured as follows. In \S 2 we give a brief survey on the basic structure of generalized flag manifolds and their invariant K\"ahler structures; we then explain the notion of index of symmetry and related geometric features. In \S 3 we give the proof of the main theorem \ref{main}.

\begin{notn*}
For a compact Lie group, we denote its Lie algebra by the corresponding lowercase gothic letter. If a group $G$ acts on a manifold $M$, for every $X\in \gg$ we denote by $\hat X$ the corresponding vector field on $M$ induced by the $G$-action.
\end{notn*}
\noindent{\bf Acknowledgments.} The author thanks D.V. Alekseevsky for valuable conversations.

\section{Preliminaries}
In this section we will review some basic facts about generalized flag manifolds and the symmetry index.
\subsection{Generalized flag manifolds} We consider a compact connected semisimple Lie group $G$ and a compact subgroup $H$ which coincides with the centralizer in $G$ of a torus. The homogeneous space $M= G/H$ is a generalized flag manifold and it can be equipped with invariant K\"ahler structures. We will now state some of the main properties of generalized flag manifolds, referring to \cite{Al,BFR} for a more detailed exposition.\par
We fix a maximal abelian subalgebra $\gt\subset\gh$ and the $B$-orthogonal decomposition $\gg = \gh \oplus \gm$. The subspace $\gm$ can be naturally identified with the tangent space $T_oM$ where $o:=[H]\in G/H$. If $R$ denotes the root system of $\gg^\bC$ relative to the Cartan subalgebra $\gt^\bC$, for every root $\alpha\in R$ the corresponding root space is given by $\gg_\alpha =\bC\cdot E_\alpha$ and
$$\gh^\bC = \bigoplus_{\alpha\in R_\gh} \gg_\alpha, \qquad \gm^\bC = \bigoplus_{\alpha\in R_\gm}\gg_\alpha,$$
where $R_\gh\subset R$ is a closed subsystem of roots and $R_\gm := R\setminus R_\gh$. The roots in $R_\gh$ are characterized by the fact that they vanish
on the center $\gc\subseteq \gt$ of $\gh$. Observe that $(R_\gh +R_\gm)\cap R \subseteq R_\gm$. \par
Any $G$-invariant complex structure $J$ on $M$ induces an endomorphism $J\in {\mathrm \End}(\gm)$ with $J^2=-Id$. If we extend $J$ to $\gm^\bC$ and we decompose
$\gm^\bC = \gm^{1,0} \oplus \gm^{0,1}$ into the sum of the $\pm i $ - eigenspaces of $J$, then the integrability of $J$ is equivalent  to the fact that
$\gq:=\gh^\bC \oplus \gm^{1,0}$ is a subalgebra, actually a parabolic subalgebra of $\gg^\bC$. Moreover it can be shown that $G$-invariant complex structures are in bijective correspondence with the invariant orderings of $R_\gm$, namely subsets $R_\gm^+\subset R_\gm$ such that :
$$R_\gm = R_\gm^+ \cup (-R_\gm^+),\quad R_\gm^+ \cap (-R_\gm^+) = \emptyset, \quad (R_\gh + R_\gm^+)\cap R \subset R_\gm^+,$$
the correspondence being given by $\gm^{1,0}= \bigoplus_{\alpha\in R_\gm^+}\gg_\alpha$. Invariant orderings are then in one-to-one correspondence with Weyl chambers in the center $\gc$ of $\gh$, namely connected components of the set $\gc\setminus \bigcup_{\alpha\in R_\gm}\ker (\alpha|_\gc)$, and an invariant ordering in $R_\gm$ can be combined with an ordering in $R_\gh$ to provide a standard ordering in $R$.\par
If we fix an invariant complex structure $J$ on $M$ (hence a Weyl chamber $C$ in $\gc$), we can endow $M$ with many $G$-invariant K\"ahler metrics which are Hermitian w.r.t. $J$. Actually, it can be proved that $G$-invariant symplectic structures, namely $G$-invariant non-degenerate closed two-forms, are in one-to-one correspondence with elements in the
Weyl chambers in $\gc$. Indeed, if $\omega\in \Lambda^2(\gm)$ is a symplectic form, then there exists $\xi$ in some Weyl chamber in $\gc$ such that
$$\omega(X,Y) = B(\ad_\xi X,Y),\quad X,Y\in \gm.$$
Moreover $\omega$ is the K\"ahler form of a K\"ahler metric $g$ w.r.t. the complex structure $J$ (i.e. $g:= \omega(\cdot,J\cdot)$ defines a K\"ahler metric) if and only if $\xi\in C$.\par
Finally if $M=G/H$ is endowed with an invariant K\"ahler structure $(g,J)$ and $G=_{\mathrm{loc}}G_1\times\ldots\times G_k$ is the decomposition into a product of simple factors, then $H$ splits accordingly as $H=_{\mathrm{loc}}H_1\times\ldots\times H_k$ for $H_i\subset G_i$ and $M$ is biholomorphically isometric to the
product of irreducible K\"ahler homogeneous spaces $M=M_1\times\ldots\times M_k$, $M_i:= G_i/H_i$. The next result, due to Onishchik (\cite{On}), deals with the basic question whether $\gg$ coincides with the full algebra of Killing vector fields.
\begin{theorem}\label{oni} If $G$ is a compact connected simple Lie group and acts almost effectively on $M=G/H$, then $G$ coincides (up to a covering) with the identity
component of the full isometry group $Q$,
with the following exceptions:
\begin{itemize}
\item[(a)] $M=\bC P^{2n+1}$ and $\gg = \sp (n+1)$, $\gh = \gu(1) \oplus \sp(n)$, $\gq = \su(2n+2)$ ;
\item[(b)] $\gg = \so(2n-1)$, $\gh = \gu(n-1)$, $\gq = \so(2n)$, \ $n\geq 4$;
\item[(c)] $M = Q_5$ and $\gg = \gg_2$, $\gh = \gu(2)$, $\gq = \so(7)$.
\end{itemize}
\end{theorem}
\subsection{The index of symmetry} In \cite{ORT} the concept of index of symmetry has been introduced. If $(M,g)$ is a Riemannian manifold and $\mathcal K(M,g)$ is the set of all Killing vector fields, at each point $x\in M$ we can define the subspace
$$\gp^x := \{X\in \mathcal K(M,g)|\ \nabla X|_x = 0\},$$
where $\nabla$ denotes the Levi Civita connection of $g$. The elements of $\gp^x$ are called {\it transvections} at $q$ and the symmetric isotropic subalgebra $\gk^x$ at $x$ is defined as the linear span of the commutators $[X,Y]$ with $X,Y\in \gp^x$. It is clear that
$$\gu^x := \gk^x \oplus\gp^x$$
is an involutive Lie algebra. The symmetric subspace $\gs_x\subset T_xM$ is then defined as
$$\gs_x :=\{\hat X_x|\ X\in \gp^x\}$$
and the {\it index of symmetry} $\imath_s(M)$ is $\inf_{x\in M}\dim \gs_x$. When $M$ is homogeneous, the assignment $x\mapsto \gs_x$ defines a distribution
which is proven to be integrable and autoparallel. The maximal integral leaves of this symmetry distribution are Riemannian symmetric spaces which are embedded in $M$ as totally geodesic submanifolds. One of the main reasons for considering the index of symmetry is the well-known fact that $\imath_s(M)= \dim M$ if and only if $M$ is a symmetric space (see \cite{ORT}).\par

\section{The main result}
In this section we consider a generalized flag manifold $M=G/H$ where $G$ is a compact semisimple connected Lie group and we endow $M$ with a $G$-invariant K\"ahler structure , given by a complex structure $J$ and a K\"ahler metric $g$. We also keep the same notations as in the previous section. \par
We will also assume that the Lie algebra $\gg$ coincides with the algebra of the full isometry group and we fix a reductive decomposition
$$\gg = \gh \oplus \gm , \qquad [\gh,\gm]\subseteq\gm.$$
We denote by $o:=[eH]\in G/H$ and by $\gp$ the subspace $\gp^o\subset \gg$. We first prove the following
\begin{lemma}\label{L1} The subspace $\gp\subseteq \gg$ is $\Ad(H)$-invariant and contained in $\gm$. Therefore $\gp \cong \gs_o$ and it is complex, i.e. $J(\gs_o)=
\gs_o$.\end{lemma}
\begin{proof} If $h\in H$ and $X\in \gp$, then $\widehat{\Ad(h)X} = h_*\hat X$ and therefore for every $w\in T_oM$ we have
$$\nabla_w\widehat{\Ad(h)X}|_o = dh^{-1}|_o(\nabla_{dh^{-1}w}\hat X) = 0,$$
hence $\Ad(h)X\in \gp$. If we fix a maximal abelian subalgebra $\gt\subseteq\gh$ then $[\gt,\gp]\subseteq\gp$. This implies that $\gp^\bC$ splits as the sum of root spaces and therefore $\gp^\bC = (\gp^\bC\cap \gh^\bC) \oplus (\gp^\bC\cap \gm^\bC)$. Since $\gp\cap \gh =\{0\}$ by well known properties of Killing vector fields, we see that $\gp^\bC \subseteq \gm^\bC$, hence $\gp\subseteq \gm$.\par
We now prove that $\gp$ is complex. We recall (see e.g.~\cite{Al}) the fact that the $\gh$-module $\gm$ splits as the sum of mutually inequivalent submodules $\gm = \bigoplus_{i=1}^k \gm_i$ each of which is therefore $J$-stable. Since $\gp$ is an $\gh$-submodule, it is the sum of a certain number of submodules $\gm_j$ and therefore it is $J$-stable.\end{proof}
Keeping the same notations, we fix a Cartan subalgebra $\gt^\bC\subset \gg^\bC$ and an ordering of the corresponding root system $R$ so that
$$\gm^{1,0} = \bigoplus_{\alpha \in R_\gm^+} \gg_\alpha.$$
The submodule $\gp^\bC$ is also the sum of certain root spaces, say  $\gp^\bC = \bigoplus_{\alpha\in R_\gp} \gg_\alpha$, where $R_\gp\subset R_\gm$ with
$R_\gp = - R_\gp$.
We now recall the well-known expression for the Levi Civita connection of an invariant metric on a reductive homogeneous space (see e.g.~\cite{KN}).
If $X,Y,U\in \gm$ then
\begin{equation}-2\ g_o(\nabla_{\hat Y}\hat X,\hat U) = \langle [Y,X]_\gm,U\rangle + \langle [X,U]_\gm,Y\rangle + \langle [Y,U]_\gm,X\rangle,\end{equation}
where $\langle \cdot,\cdot\rangle$ is the $\Ad(H)$-invariant scalar product on $\gm$ corresponding to $g_o$. Therefore $X\in \gp$ if and only if for every $Y,U\in \gm$ we have
\begin{equation}\label{conn}\langle [Y,X]_\gm,U\rangle + \langle [X,U]_\gm,Y\rangle + \langle [Y,U]_\gm,X\rangle = 0.\end{equation}
We can extend \eqref{conn} $\bC$-linearly and we can also suppose that $X = E_\alpha$ for some $\alpha\in R_\gp^+$, where $R_\gp^+=R_\gp\cap R_\gm^+$. Here $\{E_\alpha\}_{\alpha \in R}$ denotes the standard Chevalley basis of the root spaces (see e.g.~\cite{He}, p.~176). Equation \eqref{conn} implies that $E_\alpha\in \gp^\bC$ if and only if for every roots $\beta,\gamma\in R_\gm$ we have
\begin{equation}\label{conn1}\langle [E_\beta,E_\alpha]_\gm,E_\gamma\rangle + \langle [E_\alpha,E_\gamma]_\gm,E_\beta\rangle + \langle [E_\beta,E_\gamma]_\gm,E_\alpha\rangle = 0.\end{equation}
We recall now that for $\beta,\gamma\in R_\gm$ we have
$$\langle E_\beta,E_\gamma\rangle = -i \epsilon_\beta\cdot \langle JE_\beta,E_\gamma\rangle = -i\epsilon_\beta\cdot \omega(E_\beta,E_\gamma) = $$
$$-i\epsilon_\beta\cdot B([\xi,E_\beta],E_\gamma) = -i\epsilon_\beta\cdot \beta(\xi) \cdot B(E_\beta,E_\gamma)$$
where $\epsilon_\beta = \pm 1$ according to $\beta\in R_\gm^{\pm}$. Therefore $\langle E_\beta,E_\gamma\rangle = 0$ unless $\gamma= -\beta$ and
\begin{equation}\label{metric}\langle E_\beta,E_{-\beta}\rangle = -i\epsilon_\beta\cdot \beta(\xi).\end{equation}
It follows that equation \eqref{conn1} is significant only when $\beta+\gamma = -\alpha$ and in this case we have
$$N_{\beta,\gamma}\cdot \langle E_{-\alpha},E_\alpha\rangle + N_{\alpha,\gamma}\cdot \langle E_{-\beta},E_\beta\rangle +
N_{\beta,\alpha}\cdot \langle E_{-\gamma},E_\gamma\rangle = 0 .$$
Since $\alpha+\beta+\gamma = 0$ we have that $N_{\alpha,\beta}=N_{\beta,\gamma}=N_{\gamma,\alpha}$ (see \cite{He}, p.~171) and therefore we must have
\begin{equation}\label{conn2} \langle E_{-\alpha},E_\alpha\rangle = \langle E_{-\beta},E_\beta\rangle +
 \langle E_{-\gamma},E_\gamma\rangle .\end{equation}
 Using \eqref{metric} we see that \eqref{conn2} is equivalent to
 \begin{equation}\label{conn3}((1+\epsilon_\gamma)\cdot \gamma + (1+\epsilon_\beta)\cdot \beta)(\xi) = 0.\end{equation}
 Now, $\alpha>0$, so that $\beta$ and $\gamma$ are both negative or have opposite sign. If $\beta,\gamma<0$ then \eqref{conn3} is automatic, while if , say,
 $\beta<0$ and $\gamma>0$, then \eqref{conn3} implies $\gamma(\xi) = 0$. This contradicts the fact that $\gamma\in R_\gm$, while $\gh$ is the centralizer of $\xi$ in $\gg$. Therefore we conclude that $\alpha\in R_\gm^+$ belongs to $R_\gp$ if and only if $\alpha\neq -\beta-\gamma$ with $\beta$ and $\gamma$ in $R_\gm$ with opposite signs or, equivalently, if and only if $(\alpha + R_\gm^+)\cap R = \emptyset$. This allows the following characterization
 \begin{lemma} The subspace $\gp^\bC\cap \gm^{1,0}$ coincides with the center $\gz$ of the nilpotent subalgebra $\gm^{1,0}$ of $\gg^\bC$.\end{lemma}
\begin{proof} Indeed the center $\gz$ is spanned by root vectors $E_\alpha$, $\alpha\in R_\gm^+$, such that $[E_\alpha,E_\beta]=0$ for every $\beta\in R_\gm^+$ and this is equivalent to saying that $\alpha+ R_\gm^+$ does not contain roots.\end{proof}
We remark here that the subalgebra $\gm^{1,0}$ is the nilponent radical of the parabolic subalgebra $\bar\gq := \gh^\bC \oplus \gm^{1,0}$.\par
If we now put $\gk:= [\gp,\gp]$ and $\gu:= \gp \oplus\gk$, then $(\gu,\gk)$ is a symmetric pair and we can prove the following Lemma
\begin{lemma}\label{Herm} If $\gg$ is simple, the pair $(\gu,\gk)$ is an irreducible Hermitian symmetric pair. The Lie algebra $\gu$ is also simple.\end{lemma}
\begin{proof} The fact that the symmetric pair $(\gu,\gk)$  is Hermitian follows from Lemma \ref{L1}, so we need prove that it is irreducible. Note that
$\gp^\bC = \gz \oplus \bar\gz$. Moreover in \cite{BR}, p.~40, it is proved that $\ad(\gh)$ preserves $\gz$ and the action of $\gh$ on $\gz$ is irreducible. This implies that $\ad(\gh)$ acts irreducibly on $\gp$. We now decompose $\gh = \gk \oplus \gk'$ w.r.t. the Cartan Killing form $B$, where $\gk':= \gk^\perp\cap \gh$. We have
$$B([\gk', \gp],\gp) = B(\gk',[\gp,\gp]) = 0$$
and since $[\gk', \gp]\subseteq \gp$, we conclude that $[\gk', \gp] = \{0\}$. This means that the $\ad(\gk)$-action on $\gp$ is also irreducible and our first claim follows. Now $[\gu,\gu]=\gu$ so that $\gu$ is semisimple. Since $\gk$ acts on $\gp$ irreducibly and the symmetric pair is Hermitian, we immediately see that  $\gu$ is simple (see e.g.~\cite{KN}, p.~251). \end{proof}
\begin{remark}\label{dynkin} From the characterization of $\gp^\bC\cap R_\gm^+$, we see that $\gp^\bC$ contains the root space $\gg_\theta$, where $\theta\in R$ is the highest root. Actually, $\theta$ is the highest weight for the irreducible representation of $H$ on $\gz$. This gives a way for detecting the Lie algebra $\gu$ out of the painted Dynkin diagram $D$ of the flag manifold $G/H$ (see e.g.~\cite{Al,BFR} for a detailed exposition). Indeed we can consider the painted Dynkin diagram corresponding to the given flag manifold $G/H$, in which the Dynkin diagram of the semisimple part of $\gh^\bC$ is obtained by deleting some (black) nodes in the Dynkin diagram of $\gg^\bC$. We can then embed $D$ into the extended Dynkin diagram $\tilde D$ of $\gg^\bC$ and we see that the Dynkin diagram of $\gu^\bC$ is given by the connected component containing $-\theta$ of the complement in $\tilde D$ of the black nodes (see also \cite{BR}, p.~88).
\end{remark}
Now we can define the subalgebra
$$\gh' := \gh \oplus \gp $$
and note that the corresponding connected Lie subgroup $H'\subset G$ is compact, because $\gh'$ has maximal rank. Indeed the closure $\bar H'$ is connected and
has Lie algebra $\bar \gh'$ which normalizes $\gh'$. Therefore we have the $B$-orthogonal decomposition $\bar\gh' = \gh' \oplus \gf$ with $[\gh',\gf]\subseteq
\gh'\cap \gf=\{0\}$. It then follows that $\gf$ commutes with $\gh'$, hence with the maximal abelian subalgebra contained in $\gh\subset\gh'$. Therefore $\gf\subset\gh$, hence $\gf=\{0\}$ and $H'=\bar H'$.\par
The homogenous space $M' = G/H'$ is compact and has dimension given by the {\it coindex of symmetry} $k$. Note that in general the invariant complex structure on $M$ does {\it not} descend to an invariant (almost)-complex structure on $M'$.\par
\begin{lem}\label{4} We have $k\geq 6$. \end{lem}
\begin{proof} Since $G$ is simple, its action on $M'$ is almost effective and therefore $\dim G \leq \frac 12 k(k+1)$. Note that $k$ is an even integer. In \cite{BOR} it is proved that $k=2$ implies that $\dim M =3$, so that in our case $k\geq 4$. We now show that $k=4$ can be ruled out. Indeed in this case we have $\dim G \leq 10$ and being $G$ simple, we have $\gg = \su(3)$ or $\so(5)\cong\sp(2)$. In any case all the flag manifolds with isometry group $\SU(3)$ or $\SO(5)$ are Hermitian symmetric spaces (note that $\SO(5)/\U(2)\cong \SO(6)/\U(3)$).\end{proof}

\begin{lem}\label{dis} We have
\begin{equation} \dim G \leq \frac 12\ k(k-1)\end{equation}
\end{lem}
\begin{proof} We have the general estimate $\dim G \leq \frac 12\ k(k+1)$, since $G$ acts almost effectively on $M$. We first show that equality can never occur. Indeed, it is well known (see e.g.~\cite{Ko}) that equality occurs precisely when $\gg \cong \so(k+1)$ and $\gh'\cong \so(k)$. Note that $k\geq 6$ by Lemma \ref{4} and therefore $\gh'$ is simple. We note that $\gh' = \gk'\oplus \gu$ using the same notations as in the proof of Lemma~\ref{Herm}. Moreover $[\gk',\gp]=0$ implies that $[\gk',[\gp,\gp]] = [\gk',\gk] = 0$ by Jacobi, hence $[\gk',\gu]=\{0\}$ and $\gk'$ is an ideal of $\gh'$. Since $\gh'$ is simple we see that $\gk'=\{0\}$ and $\gh'=\gu\cong \so(k)$. Now using the fact that $(\gu,\gk)$ is an Hermitian symmetric pair, we have that $\gh=\gk$ is either $\mathbb R \oplus \so(k-2)$ or $\gu(\frac k2)$. Now only $\gh = \gu(\frac k2)$ can be the isotropy of a flag manifold with $\gg = \so(k+1)$ (see e.g.~\cite{BFR}) and again the only invariant K\"ahler metric on such flag manifold is the symmetric one. \par
A classical result (see e.g.~\cite{Ko}, p.47) states that the dimension $d$ of the isometry group of a $k$-dimensional manifold ($k\neq 4$) is less or equal to $\frac 12 \ k(k-1)+1$ whenever it is not equal to $\frac 12\ k(k+1)$. Moreover, when $d= \frac 12 \ k(k-1)+1$, a complete classification of the manifold is achieved (see
\cite{Ko}, p.54), showing that the isotropy representation has always a non trivial fixed vector. Since $\gh'$ has maximal rank, the isotropy representation
of $\gh'$ has no trivial submodule and therefore $\dim G \leq \frac 12\ k(k-1)$.  \end{proof}

\begin{lem}  We have that $k=6$ if and only if $\gg = \su(4)$, $\gh= 2\mathbb R\oplus \su(2)$. In this case the leaves of symmetry  are biholomorphic to $\bC P^2 = \SU(3)/\U(2)$.\end{lem}
\begin{proof} By Lemma \ref{dis} we get $\dim G \leq 15$. Since $G$ is simple and using the arguments in the proof of Lemma \ref{4}, we see that  $\gg = \su(4)$ or $\gg_2$. When $\gg = \su(4)$ we see that the only non-symmetric flag is the one with $\gh = 2\mathbb R\oplus \su(2)$ and a simple computation shows that $\gu\cong \su(3)$.\par
As for $\gg_2$, we have precisely two distinct flag manifolds with $G=G_2$, namely $M=G_2/H$ where $\gh\cong \gu(2)$ with semisimple part containing a long or a short root space. If $R_\gh$ consists of a long root, then $M\cong Q_5=\SO(7)/\SO(2)\times\SO(5)$, the metric is symmetric and $G_2$ is a proper subgroup of the full isometry group (see Theorem \ref{oni}). If $R_\gh$ is given by a short root, then $M$ is the twistor space of the Wolf space $G_2/\SO(4)$. In this case it is easy to see that the module $\gp$ is $2$-dimensional (corresponding to the fibres of the twistor fibration), hence $k= 8$.
\end{proof}
\bigskip
\bigskip
\bigskip

\vfil\eject

\bigskip\bigskip

\font\smallsmc = cmcsc8
\font\smalltt = cmtt8
\font\smallit = cmti8
\hbox{\parindent=0pt\parskip=0pt
\vbox{\baselineskip 9.5 pt \hsize=3.1truein
\obeylines
{\smallsmc
Fabio Podest\`a
Dip. di Matematica e Informatica "U.Dini"
Viale Morgagni 67/A
I-50134 Firenze
ITALY}

\medskip
{\smallit E-mail}\/: {\smalltt podesta@math.unifi.it
}
}
}

\end{document}